\documentclass[12pt]{amsart}
\usepackage{amssymb}
\usepackage{eucal}
\usepackage{amsmath}
\usepackage{amscd}
\usepackage{txfonts}      
\usepackage[dvips]{color}
\usepackage{mathtools}
\usepackage{multicol}
\usepackage[all]{xy}           
\usepackage{graphicx}
\usepackage{color}
\usepackage{colordvi}
\usepackage{xspace}
\usepackage[normalem]{ulem}
\usepackage{cancel}
\usepackage{enumitem}
\usepackage{tikz}
\usepackage{longtable}
\usepackage{multirow}
\usepackage{ifpdf}
\ifpdf
\usepackage[colorlinks,final,backref=page,hyperindex]{hyperref}
\else
\usepackage[colorlinks,final,backref=page,hyperindex,hypertex]{hyperref}
\fi

\usepackage{comment}

\newtheorem{theorem}{Theorem}[section]
\newtheorem{prop}[theorem]{Proposition}
\newtheorem{lemma}[theorem]{Lemma}
\newtheorem{coro}[theorem]{Corollary}

\newtheorem{thm-def}[theorem]{Theorem-Definition}
\newtheorem{def-prop}[theorem]{Definition-Proposition}
\newtheorem{prop-def}[theorem]{Proposition-Definition}
\newtheorem{coro-def}[theorem]{Corollary-Definition}

\theoremstyle{definition}
\newtheorem{defn}[theorem]{Definition}
\newtheorem{remark}[theorem]{Remark}

\newtheorem{example}[theorem]{Example}

\newcommand{\nc}{\newcommand}
\nc{\tred}[1]{\textcolor{red}{#1}}
\nc{\tblue}[1]{\textcolor{blue}{#1}}
\nc{\tgreen}[1]{\textcolor{green}{#1}}
\nc{\tpurple}[1]{\textcolor{purple}{#1}}
\nc{\btred}[1]{\textcolor{red}{\bf #1}}
\nc{\btblue}[1]{\textcolor{blue}{\bf #1}}
\nc{\btgreen}[1]{\textcolor{green}{\bf #1}}
\nc{\btpurple}[1]{\textcolor{purple}{\bf #1}}
\newcommand{\QSym}{\mathrm{QSym}}
\newcommand{\comp}{\vDash}

\renewcommand{\frak}{\mathfrak}

\newcommand{\delete}[1]{}

\delete{
	\nc{\mlabel}[1]{\label{#1} {{\tt {\tiny{(#1)}}}}\ }
	\nc{\mcite}[1]{\cite{#1} {{\tiny\tt (#1)}}\ }
	\nc{\mref}[1]{\ref{#1}{{\tiny\tt (#1)}}\ }
	\nc{\meqref}[1]{~\eqref{#1}{{\tiny\tt (#1)}}\ }
	\nc{\mbibitem}[1]{\bibitem[\bf #1]{#1}}
}

\nc{\mlabel}[1]{\label{#1}}  
\nc{\mcite}[1]{\cite{#1}}  
\nc{\mref}[1]{\ref{#1}}  
\nc{\meqref}[1]{~\eqref{#1}}
\nc{\mbibitem}[1]{\bibitem{#1}} 

\nc{\sbar}{, }
\delete{
{\, {\scriptstyle{|\hspace{-.08cm}|\hspace{-.08cm}|
\hspace{-.08cm}|\hspace{-.08cm}|\hspace{-.08cm}|
\hspace{-.08cm}|\hspace{-.08cm}|\hspace{-.08cm}|\hspace{-.08cm}|
\hspace{-.08cm}|\hspace{-.08cm}|\hspace{-.08cm}|
\hspace{-.08cm}|\hspace{-.08cm}|\hspace{-.08cm}|\hspace{-.08cm}|
\hspace{-.08cm}|\hspace{-.08cm}|\hspace{-.08cm}|
\hspace{-.08cm}|\hspace{-.08cm}|\hspace{-.08cm}|\hspace{-.08cm}|
\hspace{-.08cm}|\hspace{-.08cm}|\hspace{-.08cm}|
\hspace{-.08cm}|\hspace{-.08cm}|}\, }}
}

\nc{\wvec}[2]{{\scriptsize{ [
    \begin{array}{c} #1 \\ #2 \end{array}   ]}}}
\nc{\lp}{\big ( }
\nc{\llp}{\Big (}
\nc{\Llp}{\left (}
\nc{\rp}{\big ) }
\nc{\rrp}{\Big )}
\nc{\Rrp}{\right )}
\nc{\lb}{\big < }
\nc{\llb}{\!\Big \langle }
\nc{\Llb}{\! \left <}
\nc{\rb}{\big >  }
\nc{\rrb}{\Big \rangle \!}
\nc{\Rb}{\Big \rangle\! }
\nc{\length}{{\rm leng}}

\nc{\bin}[2]{ (_{\stackrel{\scs{#1}}{\scs{#2}}})}  
\nc{\binc}[2]{ \big (\! \begin{array}{c} \scs{#1}\\
    \scs{#2} \end{array}\! \big )}  
\nc{\bincc}[2]{  \left ( {\scs{#1} \atop
    \vspace{-1cm}\scs{#2}} \right )}  
\nc{\bs}{\bar{S}}
\nc{\cosum}{\sqsubset}
\nc{\la}{\longrightarrow}
\nc{\rar}{\rightarrow}
\nc{\dar}{\downarrow}
\nc{\dap}[1]{\downarrow \rlap{$\scriptstyle{#1}$}}
\nc{\uap}[1]{\uparrow \rlap{$\scriptstyle{#1}$}}
\nc{\defeq}{\stackrel{\rm def}{=}}
\nc{\disp}[1]{\displaystyle{#1}}
\nc{\dotcup}{\ \displaystyle{\bigcup^\bullet}\ }
\nc{\gzeta}{\bar{\zeta}}
\nc{\hcm}{\ \hat{,}\ }
\nc{\hts}{\hat{\otimes}}
\nc{\barot}{{\otimes}}
\nc{\free}[1]{\bar{#1}}
\nc{\uni}[1]{\tilde{#1}}          
\nc{\hcirc}{\hat{\circ}}
\nc{\lleft}{[}
\nc{\lright}{]}
\nc{\curlyl}{\left \{ \begin{array}{c} {} \\ {} \end{array}
    \right .  \!\!\!\!\!\!\!}
\nc{\curlyr}{ \!\!\!\!\!\!\!
    \left . \begin{array}{c} {} \\ {} \end{array}
    \right \} }
\nc{\longmid}{\left | \begin{array}{c} {} \\ {} \end{array}
    \right . \!\!\!\!\!\!\!}
\nc{\ora}[1]{\stackrel{#1}{\rar}}
\nc{\ola}[1]{\stackrel{#1}{\la}}
\nc{\ot}{\otimes}
\nc{\mot}{{{\sbar}}}
\nc{\otm}{\mot}
\nc{\scs}[1]{\scriptstyle{#1}}
\nc{\subv}{{^{\star}}}
\nc{\cov}{{^{\sharp}}}
\nc{\mrm}[1]{{\rm #1}}
\nc{\dirlim}{\displaystyle{\lim_{\longrightarrow}}\,}
\nc{\invlim}{\displaystyle{\lim_{\longleftarrow}}\,}
\nc{\proofbegin}{\noindent{\bf Proof: }}
\nc{\proofend}{$\quad \square$ \vspace{0.3cm}}
\nc{\sha}{{\mbox{\cyr X}}}  
\nc{\shap}{{\mbox{\cyrs X}}} 
\nc{\shpr}{\diamond}    
\nc{\shplus}{\shpr^+}
\nc{\shprc}{\shpr_c}    
\nc{\msh}{\ast}
\nc{\vep}{\varepsilon}
\nc{\labs}{\mid\!}
\nc{\rabs}{\!\mid}
\nc{\dep}{\mathrm{dep}}

\newcommand{\Q}{\mathbb{Q}}
\newcommand{\R}{\mathbb{R}}

\newcommand{\Z}{\mathbb{Z}}

\newcommand {\cala}{{\mathcal {A}}}

\newcommand {\calh}{{\mathcal {H}}}

\newcommand {\calp}{{\mathcal {P}}}

\newcommand {\calw}{{\mathcal {W}}}

\nc{\fraka}{{\frak a}}
\nc{\frakA}{{\frak A}}
\nc{\frakb}{{\frak b}}
\nc{\frakB}{{\frak B}}
\nc{\frakf}{{\frak F}}
\nc{\frakh}{{\frak h}}
\nc{\frakH}{{\frak H}}
\nc{\frakk}{{\frak k}}
\nc{\frakK}{{\frak K}}
\nc{\frakM}{{\frak M}}
\nc{\frakm}{{\frak m}}
\nc{\frakP}{{\frak P}}
\nc{\frakp}{{\frak p}}
\nc{\frakS}{{\frak S}}

\nc{\bfrakM}{\overline{\frakM}}

\nc {\e} {{\epsilon}}
\nc{\fpower}{\calp_{\rm fin}}
\nc{\pfpair}[2]{\Big(\begin{array}{c}\scs{#1} \\ \scs{#2} \end{array} \Big)}

\font\cyr=wncyr10
\font\cyrs=wncyr7

\newcommand {\calhd}{{{\mathcal {H}} _{\Z _{\ge 1}}}}

\newcommand {\Deltaa} {{\Delta_{\ge 1}}}
\nc{\id}{\mathrm{id}}
\nc{\mchi}{\chi_{{}_0}}
\nc{\Isom}{\mathrm{Isom}}
\nc{\shapb}{{\overline{\shap}}}

\nc{\mls}{{<_h}}	
\nc{\mle}{{\le_h}}
\nc{\mleq}{\leq_h}
\nc{\nmle}{{\le_m}}	
\nc{\nmleq}{{\leq_m}}

\nc {\cks}{\text{\textcircled {s}}\xspace}
\nc{\zb}[1]{\textcolor{blue}{Bin: #1}}
\nc{\xhy}[1]{\textcolor{red}{Yu: #1}}
\nc{\li}[1]{\textcolor{purple}{Li:#1}}
\begin{document}

\title[Isomorphism between Hopf algebras for MZVs]{Isomorphism between Hopf algebras for multiple zeta values }
%
\author{Li Guo}
\address{Department of Mathematics and Computer Science,
Rutgers University, Newark, NJ 07102, USA}
\email{liguo@rutgers.edu}

\author{Hongyu Xiang}
\address{School of Mathematics,
Sichuan University, Chengdu, 610064, P. R. China}
\email{xianghongyu1@stu.scu.edu.cn}

\author{Bin Zhang}
\address{School of Mathematics,
Sichuan University, Chengdu, 610064, P. R. China}
\email{zhangbin@scu.edu.cn}

\date{\today}

\begin{abstract} The classical quasi-shuffle algebra for multiple zeta values have a well-known Hopf algebra structure. Recently, the shuffle algebra for multiple zeta values are also equipped with a Hopf algebra structure. This paper shows that these two Hopf algebras are isomorphic utilizing quasi-symmetric functions.
This Hopf algebra isomorphism is compared with with the well-known isomorphism between the shuffle Hopf algebra and quasi-shuffle Hopf algebra of Hoffman, Newman and Radford.
\end{abstract}

\subjclass[2020]{
	11M32,	
	16T05,	
	12H05, 
	16T30,  
	16S10,	
	40B05	
}
\keywords{multiple zeta function, shuffle product, quasi-shuffle product, Hopf algebra, quasi-symmetric function, combinatorial Hopf algebra}

\maketitle

\vspace{-1cm}

\tableofcontents

\setcounter{section}{0}


\section{Introduction}
This paper establishes isomorphisms between the shuffle Hopf algebra and quasi-shuffle Hopf algebra encoding multiple zeta values, and compares this isomorphism with the Hopf algebra isomorphism obtained by Hoffman, Newman and Radford\cite{H2,NR}.

\subsection{The quasi-shuffle algebra and shuffle algebra for MZVs}
Multiple zeta values (MZVs) are the values of the multiple zeta series
\begin {equation}
\mlabel {eq:zeta}
\zeta (s_1, \cdots, s_k)=\sum _{n_1>\cdots > n_k>0}\frac 1{n_1^{s_1}\cdots n_k^{s_k}}
\end{equation}
at convergent points with positive integer arguments, that is, $s_1>1$, $s_i\in \Z _{\ge 1}, i=2, \cdots, k$. MZVs  and  their  generalizations  have been studied  extensively  from  different  points  of  view since 1990s with  connections  to  number theory, algebraic  geometry,  mathematical  physics,  quantum  groups  and  knot  theory \mcite{ BBV,BBBL,  BK, Car,Fur, Gon, GM,GZ1, H3, Ka, Kre, Ter, Zag, Zh}.

A fascinating aspect of MZVs is that the analytically defined values have deep algebraic structures, which have led to rich algebraic relations among MZVs and critical motivic implications, among other applications.
These algebraic structures stem from the two algebraic interpretations of the multiplications of two MZVs, namely the stuffle (or quasi-shuffle) product and shuffle product.

The stuffle relation is based on the series expression (\mref {eq:zeta}) of MZVs; while
the shuffle relation is based on the iterated integral expression of MZVs:
  {\small\begin {equation}
  \mlabel {eq:zetai}
  \zeta (s_1, \cdots s_k)=\underbrace {\int_ 0^1\frac {dt}t\int _0^t \frac {dt}t\cdots \int _0^t\frac {dt}t}_{s_1-1}\int _0^t\frac {dt}{1-t}\cdots \underbrace {\int _0^t\frac {dt}t\cdots \int _0^t\frac {dt}t}_{s_k-1}\int _0^t\frac {dt}{1-t}
  \end{equation}}
With the encoding of $\zeta(\vec s)$ by the basis element $[\vec s]$ of the graded vector space
\vspace{-.2cm}
\begin {equation}
\mlabel {eq:shafalg}
\calhd:= \Q{\bf1} \bigoplus_{k\in \Z_{>0}, \vec s\in \Z_{\ge 1}^k}\Q[\vec s],
\end{equation}
these relations are lifted as  stuffle  multiplication  $*$ and shuffle multiplication $\shapb$ on $\calhd$,
resulting the {\bf MZV quasi-shuffle algebra} $(\calhd,\ast)$ and  the {\bf MZV shuffle algebra} $(\calhd,\shapb)$.
For distinction, note that here the shuffle product $\shapb$ on $\calhd$ is not the product by shuffling the components of $[\vec s]$, but rather obtained by transporting the usual shuffle product of words on the vector space $\Q\langle x_0,x_1\rangle x_1$ underlying the noncommutative polynomial algebra, through the linear bijection
\begin {equation}
\mlabel {eq:rho0}
\rho: \calhd \to \Q \oplus\Q \langle x_0, x_1 \rangle x_1, \quad {\bf 1}\mapsto {\bf1}, [s_1, \cdots, s_k]\mapsto x_0^{s_1-1}x_1\cdots x_0^{s_k-1}x_1,
\end{equation}
with $x_0^{s_1-1}x_1\cdots x_0^{s_k-1}x_1$ encoding the iterated integral in Eq.\,\meqref{eq:zetai}.

The subspace
\begin {equation}
 \mlabel {eq:cshafalg}
 \calh ^0=\Q {\bf1} \oplus \bigoplus_{k\in \Z_{>0}, \vec s\in \Z_{\ge 1}^k, s_1>1}\Q[\vec s].
\end{equation}
is a subalgebra with  respect  to  both multiplications, and the stuffle and shuffle relations of MZVs can be viewed as a linear map
\begin {equation}
\mlabel{eq:zetahom}
\zeta: \calh ^0\to \R
\end {equation}
that preserves both multiplications $\ast$ and $\shapb$. Therefore,
$$\Big\{\zeta ([\vec s] *[\vec t])-\zeta ([\vec s] \,\shapb [\vec t])=0\ \Big| \ \vec s\in \Z_{\ge 1}^k, s_1>1, \vec t\in \Z_{\ge 1}^\ell , t_1>1, k, \ell \in \Z_{>0} \Big\}
$$
is a family of $\Q$-linear relations among MZVs, which is called the {\bf double shuffle relation}.
Ihara-Kaneko-Zagier \mcite {IKZ} extended  the algebra homomorphisms (\mref {eq:zetahom}) to algebra homomorphisms
$$\zeta^\ast, \zeta^\shap: \calhd \to \R [T],
$$
and extended the double shuffle relation to {\bf regularized double shuffle relation}:
$$\Big\{\zeta\big([1]* [\vec s]-[1]\shapb [\vec s]\big), \zeta\big([\vec s] *[\vec t]\big)-\zeta\big([\vec s] \shapb [\vec t]\big)\ \Big| \ \vec s\in \Z_{\ge 1}^k, s_1>1, \vec t\in \Z_{\ge 1}^\ell , t_1>1, k, \ell \in \Z_{>0} \Big\}.
$$
Conjecturally, they generate all algebraic relations among MZVs\,\mcite{IKZ}.

\subsection{Hopf algebra structures on the MZV quasi-shuffle algebra and the MZV shuffle algebras}
Studies of MZVs motivated the exploration of further structures on the MZV quasi-shuffle algebra $(\calhd,\ast)$ and shuffle algebra $(\calhd,\shapb)$, including their possible Hopf algebra structures.

With the deconcatenation coproduct, $(\calhd,\ast)$ becomes a Hopf algebra, as a special case of the general quasi-shuffle Hopf algebra\,\mcite{H2,NR}.
The action of this Hopf algebra on the MZV shuffle algebra has been used to obtain classes of algebraic relations of MZVs~\mcite{HO,Oh}.
Hopf algebras have played a critical role in the study of motivic MZVs~\mcite{Brown,GF,Rac}.
Such Hopf algebra structures also allow the Connes-Kreimer approach to renormalization in quantum field theory\,\mcite{CK} to be applied to studying the multiple zeta series at divergent arguments~\mcite{GZ1,MP2}. In a recent study~\mcite{GHXZ1}, a Hopf algebra structure $(\calhd,\shapb,\Deltaa)$ was obtained on the MZV shuffle algebra $(\calhd,\shapb)$ (see Theorem\,\mref{pp:unique}).

As is well-known, the graded vector space
$$ \calhd= \Q{\bf1} \bigoplus_{k\in \Z_{>0}, \vec s\in \Z_{\ge 1}^k}\Q[\vec s],
$$
has a natural shuffle Hopf algebra structure where the product $\shap$ is shuffling the vector components and the coproduct $\Delta_{\rm dec}$ is the deconcatenation as in the case of the quasi-shuffle Hopf algebra $(\calhd,\ast,\Delta_{\rm dec})$. As a special case of the Hoffman-Newman-Radford isomorphism\,\mcite{H2,NR}, there is
\begin{equation}
\mlabel{e:hoffiso1}
\exp: (\calhd,\shap,\Delta_{\rm dec}) \to (\calhd,\ast,\Delta_{\rm dec}).
\end{equation}

We note that the new Hopf algebra $(\calhd,\shapb,\Deltaa)$ has different product and coproduct from those of $(\calhd,\shap,\Delta_{\rm dec})$.
Furthermore, despite the algebra isomorphism
$ \rho: (\calhd,\shapb) \cong (\Q\oplus \Q \langle x_0, x_1 \rangle x_1,\shap)$ in Eq.\,\meqref{eq:rho0},
the coproduct $\Deltaa$ is different from the transported coproduct from the deconcatenation of the words $x_0^{s_1-1}x_1\cdots x_0^{s_k-1}x_1\in \Q \langle x_0, x_1 \rangle x_1$. In fact, the subspace $\Q \langle x_0, x_1 \rangle x_1$ is not closed under the deconcatenation operation on words. See Remark\,\mref{r:coprodnote} and Example\,\mref{ex:psiexam} for details and see\,\cite{GXZ} for a further study.

Since the Hopf algebra $(\calhd,\shapb,\Deltaa)$ is built on the MZV shuffle algebra $(\calhd,\shapb)$, it would be interesting to study its further properties, especially its relation with the MZV quasi-shuffle Hopf algebra $(\calhd,\ast,\Delta_{\rm dec})$. This paper pursues this direction and establishes the following results.
\begin{enumerate}
\item It provides a complete parametrization of Hopf algebra isomorphisms from $(\calhd, \shapb, \Deltaa)$ to $(\calhd, \ast, \Delta _{\rm dec})$ (Theorem\,\mref{t:psiisom});
\item It gives an explicit construction of such a Hopf algebra isomorphism (Theorem\,\mref{t:shachi});
\item It compares this Hopf algebra isomorphism with the Hoffman-Newman-Radford isomorphism (Theorem\,\mref{t:hoffiso}) \end{enumerate}

It is known that the algebras $(\calhd,\ast)$ and $(\calhd,\shapb)$, as well as $(\calhd,\shap)$, are polynomial algebras on countably many generators given by their respective Lyndon words. Consequently, they are isomorphic as polynomial Hopf algebras.
The significance of our results lies in providing a natural and explicit construction of such an isomorphism, in close analogy with the Hoffman-Newman-Radford Hopf algebra isomorphism. In view of the fundamental roles played by the MZV shuffle algebra $(\calhd,\shapb)$ and the MZV quasi-shuffle algebra $(\calhd,\ast)$ in the study of MZVs, the Hopf algebra isomorphism obtained here is hoped to provide a new perspective.

\subsection{Outline and notations}
The paper is structured as follows. In Section \mref{s:hopfhom}, we recall the construction of the Hopf algebra structure on the shuffle algebra $(\calhd,\shapb)$ and compare it with other Hopf algebras on $\calhd$. We then apply the universal property of the Hopf algebra of quasi-symmetric functions~\mcite{ABS06} to obtain Hopf algebra homomorphisms from  $(\calhd, \shapb, \Deltaa)$ to $(\calhd, \ast, \Delta _{\rm dec})$ that are induced by characters on $(\calhd, \shapb)$ (Corollary~\mref{coro:IndHom}).

Section \mref{s:hopfiso} determines when such a Hopf algebra homomorphism is an isomorphism. A well-order is defined on the homogeneous components of $\calhd$ and is extended to a transitive relation on tensor powers of $\calhd$. This relation allows us to obtain a bijection between the set of Hopf algebra isomorphisms from   $(\calhd, \shapb, \Deltaa)$ to $(\calhd, \ast, \Delta _{\rm dec})$ and the set of characters on  $(\calhd, \shapb)$ that is nonzero on every depth one vector (Theorem\,\mref{t:psiisom}). Finally one such character is constructed, showing that the MZV shuffle Hopf algebra $(\calhd,\shapb,\Deltaa)$ and the MZV quasi-shuffle algebra $(\calhd,\ast,\Delta_{\rm dec})$ are isomorphic. Combining with the Hoffman-Newman-Radford isomorphism, we conclude that all the three Hopf algebras $(\calhd,\ast,\Delta_{\rm dec}), (\calhd,\shap,\Delta_{\rm dec})$ and $(\calhd,\shapb,\Deltaa)$ are naturally isomorphic. Examples are provided for the computations of these isomorphisms.

\noindent
{\bf Notations.}
For a set $X$, let
\begin{equation}
H_X^+:=\bigcup _{k\in \Z _{>0}}X^k,\quad  H_X:=\{{\bf1}\}\cup \bigcup _{k\in \Z _{>0}}X^k,
\mlabel{eq:hset}
\end{equation}
and $\calh ^+_X$, $\calh _X$ be the vector spaces over $\Q $ with a basis $H_X^+$ and $H_X$,
that is
\begin{equation}
\calh ^+_X :=\Q H_X^+=\bigoplus _{k\in \Z _{>0}, \vec s\in X ^k }\Q [\vec s], \quad \calh _X :=\Q H_X=\Q {\bf1}\oplus \bigoplus _{k\in \Z _{>0}, \vec s\in X ^k }\Q [\vec s].
\mlabel{eq:hspace}
\end{equation}
Here we use $[\vec s]$ to denote the basis element without any algebraic structure inherited from $X$.
We are mostly interested in the case of $\Z _{\ge 1}$. For $\vec s=(s_1,\cdots, s_k)$ $\in \Z_{\ge 1}^k$, $k$  is  called  the {\bf depth}  of $[\vec s]$  and  $|\vec s|=s_1+\cdots+s_k$  is called the {\bf weight} of $[\vec s]$.

\section {The shuffle Hopf algebra and quasi-shuffle Hopf algebra for MZVs}
\mlabel{s:hopfhom}
This section recalls three Hopf algebraic structures with the common underlying vector space $\calhd$. Two of them follows from the general constructions of shuffle Hopf algebras and quasi-shuffle Hopf algebras on a set or a commutative semigroup by the deconcatenation coproduct.
The quasi-shuffle algebra on $\calhd$ encodes the stuffle product of MZVs; while the shuffle algebra is not the MZV shuffle algebra $(\calhd,\shapb)$. These two Hopf algebras are isomorphic by a classical result~\mcite{H2,NR}. The third Hopf algebra on $\calhd$ is obtained by transporting of structures from a Hopf algebra on another vector space where the product encodes the shuffle product of MZVs, and the coproduct $\Deltaa$ was recently discovered\,\mcite{GHXZ1}. The purpose of this section is establish the connection between the MZV shuffle Hopf algebra $(\calhd,\shapb,\Deltaa)$ and the MZV quasi-shuffle algebra $(\calhd,\ast,\Delta_{\rm dec})$.

\subsection{Shuffle and quasi-shuffle Hopf algebras by deconcatenation}

We first recall the usual construction of shuffle and quasi-shuffle Hopf algebras where the coproduct is by deconcatenation.

For an alphabet set $\cala$, let $\calw _\cala $ be the set of words (including the empty word $\bf 1$), namely the free monoid on $\cala$. Then in the linear space $\Q \calw _\cala $ with a basis $\calw _\cala$, the {\bf shuffle product} $\shap$ is defined by the recursion
\begin{equation}
\mlabel{e:shufprod}
w\shap {\bf 1}={\bf 1}\shap w=w, \quad
a w_1\shap bw_2=a(w_1\shap b w_2)+b(a w_1\shap w_2),
\end{equation}
for $w, w_1, w_2 \in \calw_\cala $, $a, b\in \cala$.
Furthermore, define the {\bf deconcatenation coproduct} on $\Q\calw_\cala$ by
\begin{equation}
\Delta_{\rm dec}({\bf 1})={\bf 1}\otimes {\bf 1},\quad
\Delta_{\rm dec}(w)=\sum_{w_1w_2=w} w_1\otimes w_2,
\mlabel{e:deccop}
\end{equation}
and the linear map
\begin{align}
	\epsilon: \Q\calw_\cala \to \Q, \quad  \epsilon(w)=\left\{\begin{array}{ccc}
		&1,& w={\bf 1},\\
		&0,& w\in \calw_\cala, \not ={\bf 1}.
	\end{array}\right.
\mlabel{e:deccounit}
\end{align}
The resulting Hopf algebra $(\Q\calw_\cala, \shap ,{\bf1}, \Delta_{\rm dec}, \epsilon)$ is called the {\bf shuffle Hopf algebra with alphabet $\cala$}.

\begin{example} For the alphabet $\cala=\Z_{\ge 1}$, we use a vector in place of a word. Then
$$\Q\calw _\cala =\calhd
$$
in Eq.\,\meqref{eq:shafalg}, and
we have the shuffle Hopf algebra $(\calhd, \shap, {\bf 1},\Delta_{\rm dec}, \epsilon)$.
\end{example}

If the alphabet $\cala$ is a commutative semigroup with a product $\cdot$, then $\Q\calw _\cala $ has a {\bf quasi-shuffle product} $\ast$ defined by the recursion
\begin{equation}
\mlabel{e:qshufprod}
\begin{split}
w\ast {\bf 1}&={\bf 1}\ast w=w,\\
a w_1\ast bw_2&=a(w_1\ast b w_2)+b(a w_1\ast w_2) +(a\cdot b)(w_1\ast w_2).
\end{split}
\end{equation}
With the same deconcatenation coproduct \meqref{e:deccop} and the counit in \meqref{e:deccounit}, we obtain a Hopf algebra
$(\Q\calw _\cala, \ast, {\bf1}, \Delta_{\rm dec}, \epsilon )$, called the {\bf quasi-shuffle Hopf algebra (on the semigroup alphabet $\cala$)}.

\begin{example} Taking the additive semigroup $\cala=\Z_{\ge 1}$, we have the MZV quasi-shuffle Hopf algebra $(\calhd, \ast, \Delta_{\rm dec})$.
\mlabel{ex:quasishuf}
\end{example}

As a special case of the {\bf Hoffman-Newman-Radford isomorphism}\,\cite{AM,H2,NR}, there is a Hopf algebra isomorphism
\begin{equation}
\begin{split}
\exp:&  (\calhd,\shap,\Delta_{\rm dec})\to (\calhd,\ast,\Delta_{\rm dec}), \\
{\bf1}&\mapsto {\bf1}, \\
[s_1,\ldots,s_k]& \mapsto
\sum_{(i_1,\ldots,i_\ell)\models k} \frac{1}{i_1!\cdots i_\ell!} [s_1+\ldots+s_{i_1},s_{i_1+1}+\ldots+s_{i_2},\ldots ,s_{i_1+\cdots+i_{\ell-1}+1}+\ldots+s_k].
\end{split}
\mlabel{e:hoffexp}	
\end{equation}
The inverse isomorphism is
\begin{equation}
\begin{split}
\log:& (\calhd,\ast,\Delta_{\rm dec}) \to (\calhd,\shap,\Delta_{\rm dec}), \\
{\bf1}&\mapsto {\bf1}, \\
[s_1,\ldots,s_k]& \mapsto
\sum_{(i_1,\ldots,i_\ell)\models k} \frac{(-1)^{k-\ell}}{i_1!\cdots i_\ell!} [s_1+\ldots+s_{i_1},s_{i_1+1}+\ldots+s_{i_2},\ldots, s_{i_1+\cdots+i_{\ell-1}+1}+\ldots+s_k].
\end{split}
\mlabel{e:hofflog}	
\end{equation}

\subsection{The shuffle Hopf algebra for MZVs}
\mlabel{ss:shufhopf}
The third Hopf algebraic structure on $\calhd$ arises from the integral interpretations of MZVs in Eq.\,(\mref {eq:zetai})\,\mcite{GHXZ1}.

For the alphabet $\cala=\{x_0,x_1\}$, the shuffle algebra $\Q \calw_\cala$ naturally shares its underlying space with the noncommutative polynomial algebra $\Q\langle x_0,x_1\rangle$, of which the subspace $x_0 \Q\langle x_0,x_1\rangle x_1$ spanned by words started with $x_0$ and ended with $x_1$ is closed under the word shuffle product $\shap$ in Eq.\,\meqref{e:shufprod}.

Through the linear bijection
\begin {equation}
\mlabel {eq:rho}
\rho: \calhd \to \Q\oplus \Q \langle x_0, x_1 \rangle x_1, \ {\bf1}\mapsto {\bf1}, \ [s_1, \cdots, s_k]\mapsto x_0^{s_1-1}x_1\cdots x_0^{s_k-1}x_1,
\end{equation}
the word shuffle product $\shap$ on the right is pull back to  a multiplication $\shapb$ on $\calhd$, and the algebra $(\calhd,\shapb)$ is called the {\bf MZV shuffle algebra}. Then
$$\rho ^{-1}(x_0\Q \langle x_0, x_1 \rangle x_1)=\bigoplus _{k\in \Z _{>0}, \vec s\in \Z _{\ge 1} ^k, s_1\ge 2 }\Q [\vec s]
$$
is a (nonunitary) subalgebra corresponding to the MZVs.

For further details of the product $\shapb$ including its intrinsic description without referring to $\rho$ and its relation with the Rota-Baxter operator, we refer the reader to \mcite{BBBL,GX,GZ3}. The case of $[s]\shapb [t]$ for scalars $s, t\in \Z_{\geq 1}$ corresponds to Euler's decomposition formula for $\zeta(s)\zeta(t)$.

We next recall the coproduct $\Deltaa$ on $\calhd$ which is defined by a family of linear operators \mcite {GHXZ1}:
{\small
\begin{equation}
\begin{split}
	\delta_i: \calhd& \longrightarrow \calhd,\quad i\ge 1,\\
{\bf1}&\mapsto 0, \\
	[s_1,\cdots,s_k]& \mapsto \left\{\begin{array}{ll}
			\sum_{j=1}^i s_j[s_1,\cdots, s_j+1,\cdots, s_i,\cdots,s_k], & i\le k,  \\
			0,& i>k.
		\end{array}\right.
\end{split}
\mlabel{eq:deltadef}
\end{equation}
}

Let
$$p_i:=\delta _i-\delta _{i-1}.
$$
Then by\,\cite[Lemma\,2.2]{GHXZ1}, we have

\begin{align}
	p_{i} [s_1,\ldots,s_k]&=s_i [s_1,\ldots,s_{i-1},s_i+1,s_{i+1},\ldots,s_k], \quad 1\leq i\leq k,
	\mlabel{eq:ptos1}\\
	p_{k+1} [s_1,\ldots,s_k]&=-\sum_{j=1}^k s_j[s_1,\cdots, s_j+1,\cdots, s_k],\mlabel{eq:ptos2} \\
	p_{i} [s_1,\ldots,s_k]&=0, \quad i\leq 0 \text{ or } i\geq k+2.
	\mlabel{eq:ptos3}
\end{align}
Furthermore,
\begin{equation}
\mlabel{eq:Fraction}
[s_1, \cdots, s_i, \cdots, s_k]=\frac {p_1^{s_1-1}\cdots p_k^{s_k-1}}{(s_1-1)!\cdots (s_k-1)!}[1_k].
\end{equation}
Here we use the abbreviation $[1_k]=[\underbrace{1,\ldots,1}_{k}]$.

For a graded linear operator $A: \calhd \to \calhd$,  the {\bf shifted tensor product} of $A$ with the operator sequence $\{\delta _i\}_{i \geq 1}$ in Eq.\,\meqref{eq:deltadef} is defined by
\begin{equation}
	\begin{split}
		A\cks \delta_i: \calhd\otimes \calhd&\longrightarrow \calhd\otimes \calhd, \\
		[\vec s]\otimes [\vec t]&\mapsto
		A([\vec s])\otimes \delta_{i-\dep(\vec s)}([\vec t]).
	\end{split}
	\mlabel{eq:shiftten}
\end{equation}

\begin{prop}\mcite {GHXZ1}
There is a unique linear operator $\Deltaa : \calhd\longrightarrow \calhd\otimes \calhd$  such that
\begin{enumerate}
\item  $\Deltaa({\bf 1})={\bf 1}\otimes {\bf 1}$;
\item $\Deltaa([1_k])=\sum\limits_{j=0}^{k}[1_j]\otimes [1_{k-j}]$;
\item  for any $i\in \Z _{\ge 1}$,
$$(\id\ \cks \delta_i+\delta_i\otimes \id)\Deltaa=\Deltaa \delta_i.
$$
\end{enumerate}
Furthermore, $\Deltaa$ is coassociative and is a homomorphism with respect to $\shapb$.
\mlabel{pp:unique}
\end{prop}
To illustrate the new coproduct, we give some examples.
\begin{align*}
	\Deltaa([1,2])&={\bf 1}\otimes [1,2]+[1]\otimes [2]-[2]\otimes [1]+[1,2]\otimes {\bf 1}, \\
	\Deltaa([2,2])&={\bf 1}\otimes [2, 2]+[2]\otimes [2]-2[3]\otimes [1]+[2,2]\otimes {\bf 1}, \\
	\Deltaa([1,2,1])&={\bf 1}\otimes [1,2,1]+[1]\otimes [2,1]-[2]\otimes [1,1]+[1,2]\otimes [1]+[1,2,1]\otimes {\bf 1}.
\end{align*}

The space $\calhd$ has a grading by weight:
\begin{equation}
	\calhd=\bigoplus_{k=0}^{\infty}\Q H_k, \mlabel{eq:calhdgrading}
\end{equation}
where
$$H_0:={\bf 1},
\quad H_k:=\Big\{[s_1,\cdots,s_m]\in\Z_{\ge 1}^{m}\,\Big|\,s_1+\cdots+s_m=k\Big\}.
$$
Together with the unit $u_{\ge 1}: \Q\to\calhd , 1\mapsto {\bf 1}$, counit
$\varepsilon_{\geq 1}: \calhd\longrightarrow \Q, {\bf 1}\mapsto 1, [\vec s]\to 0$
and the weight grading, we have the following result.

\begin{prop} \mcite{GHXZ1}
	$(\calhd, \shapb, u_{\ge 1}, \Deltaa, \varepsilon_{\geq 1})$ is a connected graded Hopf algebra.
\end{prop}

We elaborate on the difference between this Hopf algebra and the existing shuffle Hopf algebras with deconcatenation coproducts on $\calhd$ or on $\Q\langle x_0,x_1\rangle$.

\begin{remark}
\mlabel{r:coprodnote}
\begin{enumerate}
\item First the coproduct $\Deltaa$ is not the coproduct $\Delta_{\rm dec}$ that deconcatenates the vectors $[s_1,\cdots,s_k]$ of the underlying vector space
$$ \calhd= \Q{\bf1} \bigoplus_{k\in \Z_{>0}, \vec s\in \Z_{\ge 1}^k}\Q[\vec s].
$$
For example,
\begin{equation}
 	\Deltaa([1,2])={\bf 1}\otimes [1,2]+[1]\otimes [2]-[2]\otimes [1]+[1,2]\otimes {\bf 1},
\mlabel{e:shufexam}
\end{equation}
which has the term $-[2]\otimes [1]$ not appearing in the deconcatenation coproduct.
\item
Next the coproduct $\Deltaa$ is not the one that comes from deconcatenating the words in $\Q\langle x_0,x_1\rangle x_1$ and then being transported by the linear isomorphism
$$
\rho: \calhd \to \Q\oplus\Q \langle x_0, x_1 \rangle x_1, \ {\bf1}\mapsto {\bf1}, [s_1, \cdots, s_k]\to x_0^{s_1-1}x_1\cdots x_0^{s_k-1}x_1
$$
in Eq.\,\meqref{eq:rho}. In fact, the subspace
$\Q \langle x_0, x_1 \rangle x_1$ of $\Q\langle x_0,x_1\rangle$ is not closed under the word deconcatenation coproduct.
For example, word deconcatenation gives
$$ \Delta_{\rm dec}(x_1x_0x_1)= {\bf 1} \ot x_1x_0x_1+x_1\ot x_0x_1+x_1x_0\ot x_1+x_1x_0x_1\ot {\bf 1},$$
where $x_1x_0$ is not in $\Q \langle x_0, x_1 \rangle x_1$.
In our case, $x_1x_0x_1$ corresponds to $[1,2]$. The new coproduct $\Deltaa$ on $\calhd$ transports to a coproduct $\tilde{\Delta}_{\ge 1}$ on $\Q \langle x_0, x_1 \rangle x_1$ for which the coproduct in Eq.\,\meqref{e:shufexam} corresponds to
$$ \tilde{\Delta}_{\ge 1} (x_0x_1x_0)= {\bf 1}\ot x_0x_1x_0 + x_1\ot x_0x_1 - x_0x_1\ot x_1 + x_1x_0x_1\ot {\bf 1}.$$
\item
Since our Hopf algebra $(\calhd,\shapb,\Deltaa)$ is not the Hopf algebra $(\calhd,\sha,\Delta_{\rm dec})$, the Hopf algebra isomorphism obtain in this paper (Theorem\,\mref{t:shachi}) is different from the Hoffman-Newman-Radford isomorphism in Eq.\,\meqref{e:hoffexp}\,\mcite{H2,NR}. The relation between these two Hopf algebra isomorphims will be given in Theorem\,\mref{t:hoffiso}.
\end{enumerate}
\end{remark}

\subsection{Hopf algebras of quasi-symmetric functions, of MZV quasi-shuffles, and of MZV shuffles}
The notion of quasi-symmetric functions was formally introduced by Gessel~\mcite{Ge} in 1984 with its motivation traced back to the work of Stanley~\mcite{St1972} on $P$-partitions and of MacMahon~\mcite{Mac} on plane partitions.

A natural linear basis of the space $\QSym$ of quasi-symmetric functions is the monomial quasi-symmetric functions
\begin{equation} \mlabel{eq:mqsym}
M_\alpha:=\sum_{0<i_1<i_2<\cdots<i_k}x_{i_1}^{\alpha_1}x_{i_2}^{\alpha_2}\cdots x_{i_k}^{\alpha_k}
\end{equation}
parameterized by compositions $\alpha=(\alpha_1,\cdots,\alpha_k)$  of \(n\) (denoted \(\alpha \comp n\)).

The space $\QSym$ has a natural graded Hopf algebra structure.
Here the grading is given by
$$ \QSym=\bigoplus_{n\ge 0} \QSym_n,$$
where $\QSym_n$ is the linear space spanned by all monomial quasi-symmetric functions $M_\alpha$ with $\alpha \comp n$.
The product in $\QSym$ is the usual multiplication of formal power series. With respect to the linear basis of monomial quasi-symmetric functions, there is the relation
\[
M_\alpha  M_\beta = M_{\alpha \ast \beta}
\]
where $\alpha \ast \beta$ is the quasi-shuffle product of the compositions $\alpha $ and $\beta$ defined by the same recursion as in Example~\mref{ex:quasishuf}.

The coproduct \(\Delta_{\rm dec}: \QSym \to \QSym \otimes \QSym\) is given by deconcatenation of compositions:
\[
\Delta_{\rm dec}(M_{(a_1, \dots, a_k)}) = \sum_{i=0}^{k} M_{(a_1, \dots, a_i)} \otimes M_{(a_{i+1}, \dots, a_k)}.
\]
with the counit \(\epsilon: \QSym \to \mathbb{Q}\) being the projection onto the constant term: \(\epsilon(M_\alpha) = 1\) if \(\alpha=\emptyset\), and \(0\) otherwise.

Therefore, the resulting Hopf algebra $(\QSym,\cdot,\Delta_{\rm dec})$ is isomorphic to the quasi-shuffle Hopf algebra $(\calhd,\ast,\Delta_{\rm dec})$~\mcite{H2}:
$$\Phi: \QSym \to \calhd, M_\alpha \mapsto [\vec \alpha],$$
and the isomorphism preserves grading.

A remarkable property of quasi-symmetric functions is that their Hopf algebra structure is the terminal object in the category of combinatorial Hopf algebras~\mcite{ABS06}.
Recall \mcite{ABS06} that a {\bf combinatorial Hopf algebra} over $\Q$ is a pair $(H, \chi)$, where $H$ is a graded connected Hopf algebra over $\Q$ and $\chi: H\rightarrow \Q$ is an algebra homomorphism, called a character.
For the Hopf algebra $\QSym$, the map
\begin{equation}
	\mlabel{e:charqs}
\begin{split}
\chi_Q :& \QSym \longrightarrow \Q, \\
\chi_{Q} (M_\emptyset)&=1, \quad \chi_Q(M_{(n)})=1, n\geq 1, \\ \chi_Q(M_\alpha)&=0, \text{ otherwise}.
\end{split}
\end{equation}
defines a character.

\begin{theorem}\cite[Theorem 4.1]{ABS06}
The combinatorial Hopf algebra $(\QSym, \chi_Q)$ is the terminal object in the category of combinatorial Hopf algebras.
\mlabel{prop:Qsymisterminalobject}
\end{theorem}

Due to the graded Hopf algebra isomorphism $\Phi$ in Eq.\,\meqref{e:charqs}, the universal property of the Hopf algebra $\QSym$ can be reformulated in terms of the Hopf algebra $(\calhd,\ast,\Delta_{\rm dec})$ as follows.

\begin{coro}
\mlabel{coro:IndHom}
With the character
\begin{equation}
	\mlabel{eq:zetaQdefinitionQSym}
	\begin{split}
		\chi_Q'=\chi_Q\circ \Phi^{-1}:& \calhd\to \Q, \\
		\chi_Q'({\bf1})=1,& \ \chi_Q'([s_1])=1, \ \chi_Q'([s_1, \cdots, s_k])=0, k>1.
	\end{split}
\end{equation}
the combinatorial Hopf algebra $(\calhd,\ast,\Delta_{\rm dec},\chi_Q')$ is a terminal object in the category of combinatorial Hopf algebras. More precisely, for any combinatorial Hopf algebra $(H,\cdot,\Delta,\chi)$, there is a unique Hopf algebra homomorphism
\begin{equation}
\mlabel{e:psiform}
\Psi _\chi :(H,\cdot,\Delta)\to (\calhd, \ast, \Delta _{\rm dec})
\end{equation}
such that $\chi'_Q \circ \Psi_\chi=\chi$.
\end{coro}

\begin{example}\mlabel{exm:Hoffiso}
 The Hopf algebra isomorphism
$$\exp: (\calhd,\shap,\Delta_{\rm dec})\to (\calhd,\ast,\Delta_{\rm dec})
$$
in Eq.\,\meqref{e:hoffexp} can be viewed as a special case of the universal property of  $(\calhd,\ast,\Delta_{\rm dec})$ corresponding to the character
\begin{equation}
	\mlabel{e:charh}
\chi_H: (\calhd, \shap) \to \Q, \quad {\bf1}\mapsto 1; \  [s_1, \cdots, s_k]\mapsto \frac 1{k!}.
\end{equation}
\end{example}

\section{Isomorphisms between Hopf algebras for MZVs}
\mlabel{s:hopfiso}
This section studies properties of the Hopf algebra homomorphisms $\Psi_\chi: (\calhd, \shapb, \Delta_{\ge 1})\to (\calhd, \ast, \Delta_{\rm dec})$ obtained in Corollary \mref{coro:IndHom} and determines the condition for the homomorphisms to be isomorphisms (Theorem\,\mref{t:chiisom}). For this purpose, we first introduce an order on the basis of the homogeneous components of $\calhd$, so that the matrix of $\Psi_\chi$ with respect to this basis is upper triangular (Proposition\,\mref{prop:Psiisomorphism}). Since the construction of $\Psi_\chi$ in \meqref{eq:chialpha} involves iterated coproducts, we also need an order or relation on the tensor powers of these homogeneous components. These will be provided in Section~\mref{ss:order}.

\subsection{An order for tensors}
\mlabel{ss:order}

Consider the grading
$$\calhd =\bigoplus _{n\ge 0} \Q H_n
$$
given in Eq.\,\meqref{eq:calhdgrading}. We define a relation $\mls$ on $H_n$ by restricting the lexicographic order to vectors of the same weight: for $[s_1,\cdots, s_m]$, $[t_1, \cdots, t_\ell ]\in H_n$, define
\begin{equation}
[s_1, \cdots, s_m]\mls[t_1, \cdots, t_\ell]
\mlabel{e:order1}
\end{equation}
if there is $1\le j\le \min\{m, \ell\}$ such that $s_1=t_1, \cdots, s_{j-1}=t_{j-1}$, $s_j>t_j$. As usual, define $[\vec s]\mle [\vec t]$ if $[\vec s]\mls [\vec t]$ or $[\vec s]=[\vec t]$.
This order is obtained from transporting the lexicographic order in $\Q\langle x_0,x_1\rangle x_1$ with $x_0<x_1$ which is often used in the MZV literature\,(see e.g. \cite{BJOP,GPZ,H3},\cite[\S\,3.3.1]{Zh}).

\begin{example}
For the basis elements of $H_4$, we have $$[4]\mls[3,1]\mls[2,2]\mls[2,1,1]\mls[1,3]\mls[1,2,1]\mls[1,1,2]\mls[1,1,1,1].$$
\end{example}

Here are some simple properties of $\mls$.

\begin{lemma}
\mlabel{lem:extension}
\begin{enumerate}
\item
The relation $\mls$ is a well order.
\mlabel{i:ext1}
\item
The maximal element in $H_n$ is $[1_{n}]$, the minimal element in $H_n$ is $[n]$;
\mlabel{i:ext2}
\item
If $[\vec u]\mls[\vec v]$, then
$$[\vec u, \vec w]\mls[\vec v, \vec w], \ [\vec w, \vec u]\mls[\vec w, \vec v].
$$
\mlabel{i:ext3}
\end{enumerate}
\end{lemma}
\begin{proof}
\meqref{i:ext1} It follows from the fact that if $s_i=t_i$ for all $1\le i\le \min\{m,\ell\}$, then $m=\ell$ and $\vec s=\vec t$.

\meqref{i:ext2} and
\meqref{i:ext3} follow directly from the definition.
\end{proof}

We next extend the well order $\mle$ on the homogeneous pieces of $\calhd$ to a transitive relation $\nmle$ on tensor powers of $\calhd$.

\begin{defn}\mlabel{def:orderten}
\begin{enumerate}
\item A pure tensor $[\vec u_1]\otimes \cdots \otimes [\vec u_k]\in (\calhd)^{\otimes k}$ is called a {\bf homogeneous pure tensor} if $\vec u_i\in \Z _{\ge 1}^{s_i}$, $i=1,\cdots, k$. Thus an element of $(\calhd)^{\otimes k}$ is a finite sum of homogeneous pure tensors.
\item The {\bf weight} of a homogeneous pure tensor $[\vec u_1]\otimes \cdots \otimes [\vec u_k]$ is defined to be the weight of the vector $[\vec u_1,\cdots, \vec u_k]$, namely the sum of all the components of the vector.
\item
For homogeneous pure tensors $[\vec u_1]\otimes \cdots \otimes [\vec u_k]\in (\calhd)^{\otimes k}$ and $[\vec v_1]\otimes \cdots \otimes [\vec v_\ell]\in (\calhd)^{\otimes \ell}$
{\it of the same weight}, define
\begin{equation}\mlabel{e:orderh}
[\vec u_1]\otimes \cdots \otimes [\vec u_k]\nmle  [\vec v_1]\otimes \cdots \otimes [\vec v_\ell]
\end{equation}
if
	$$[\vec u_1,\cdots, \vec u_k]\mle [\vec v_1,\cdots,\vec v_\ell].
	$$
Only homogeneous pure tensors with the same weight can be compared under $\nmle$.
\item
More generally, define a relation $\nmle$ on $\bigoplus_{k\geq 0} (\calhd) ^{\otimes k}$  by
\begin{equation}
\mlabel{e:order2}
\sum _ia_i [\vec u_{i,1}]\otimes \cdots \otimes [\vec u_{i,k_i}]\nmle \sum _j b_j[\vec v_{j,1}]\otimes \cdots \otimes [\vec v_{j,\ell_j}]
\end{equation}
if for each $i,j$ with $a_i\neq 0$ and $b_j\neq 0$, the corresponding homogeneous pure tensors have the same weight and
	$$[\vec u_{i,1}]\otimes \cdots \otimes [\vec u_{i,k_i}]\nmle [\vec v_{j,1}]\otimes \cdots \otimes [\vec v_{j,\ell_j}].
	$$
\end{enumerate}
\end{defn}

\begin{example}
	\begin{enumerate}
		\item ${\bf1}\otimes [1]\nmle [1]$.
		\item $[2]\otimes [2]\nmle [1,3]$;
		\item $[2]\otimes [2]\nmle [2,2]$ and $[2,2]\nmleq [2]\otimes [2]$;
		\item $[4]-[3,1]\nmle [2,2]$.
	\end{enumerate}
\end{example}

\begin{lemma}
	\mlabel{lem:pkeepsorder} The operators $p_i$ preserves the order $\nmle$ in the sense that, if $[\vec t]\nmle [\vec s]$, then for $0<i\le k=\dep([\vec s])$, either $p_i ([\vec t])\not =0$ or
	$$p_i([\vec t])\nmle p_i([\vec s]).
	$$
\end{lemma}
\begin{proof} Let $\vec t=(t_1, \cdots, t_m)$ and $\vec s=(s_1, \cdots, s_k)$. Since $[\vec t]\nmle [\vec s]$, by definition,
	either
	$$k=m, [t_1, \cdots, t_{m}]= [s_1, \cdots, s_k],
	$$
	or
	there is $1\le i_0\le m$, such that
	$$t_1=s_1, \cdots, t_{i_0-1}=s_{i_0-1}, t_{i_0}>s_{i_0}.$$
	
	We verify the conclusion by separately considering three cases.
	\begin{itemize}
		\item [{\bf Case 1:}] Let $0<m+1<i$. Then by Eq.\,\meqref{eq:ptos3},
		$$p_i([\vec t])=0,$$
as desired.
		\item [{\bf Case 2:}] Let $i=m+1$. In this case $[\vec t]\not =[\vec s]$.
		So
		there is $1\le i_0\le m$, such that
		$$t_1=s_1, \cdots, t_{i_0-1}=s_{i_0-1}, t_{i_0}>s_{i_0}.$$
		Then by Eqs.\,\meqref{eq:ptos1} and \meqref{eq:ptos2}, we obtain
		\begin{equation*}
			p_i([\vec t])=-\sum _jt_j[t_1, \cdots, t_j+1, \cdots, t_{m}]
		\end{equation*}
		and
		$$p_i([\vec s])=s_{m+1}[s_1, \cdots, s_{m}, s_{m+1}+1, \cdots, s_k].
		$$
		Therefore,
		$$[t_1, \cdots, t_j+1, \cdots, t_{m}]\mls [s_1, \cdots, s_{m}, s_{m+1}+1, \cdots, s_k]$$
		for each $1\le j\le m$,
		and  so $$p_i([\vec t])\nmle p_i([\vec s]).
		$$
		\item [{\bf Case 3:}] Let $0<i\le m$. Then Eq.\,\meqref{eq:ptos1} gives
		$$p_i([\vec t])=t_i[t_1, \cdots, t_i+1, \cdots, t_{m}]
		$$
		and
		$$p_i[\vec s]=s_i[s_1, \cdots, s_i+1, \cdots, s_{k}].
		$$
		Thus
		$$\hspace{6.4cm} p_i([\vec t])\nmle p_i([\vec s]).
		\hspace{6.4cm} \qedhere $$
	\end{itemize}
\end{proof}

Lemma~\mref{lem:extension} gives
\begin{lemma}
\begin{enumerate}
\item We have $[\vec u_1, \ldots, \vec u_k]\nmle [\vec u_1]\ot \cdots \ot [\vec u_k] \nmle [\vec u_1,\ldots, u_k]$;
\mlabel{i:ten1}
\item The relation $\nmle$ in Eq.~\meqref{e:order2} is transitive.
\mlabel{i:ten3}
\item If $[\vec u_i]\nmle [\vec v_i]$\, for $i=1,\ldots,k$, then
$[\vec u_1] \ot \cdots \ot [\vec u_k]\nmle [\vec v_1]\ot \cdots \ot [\vec v_k].$
\mlabel{i:ten2}
\end{enumerate}
	\mlabel{lem:exttensor}
\end{lemma}
\begin{proof}
\meqref{i:ten1} Follows directly from Definition \mref{def:orderten}.

\meqref{i:ten3}. The transitivity of $\nmle $ in the general case in Eq.\,\meqref{e:order2} follows from the case for homogeneous pure tensor in Eq.\,\meqref{e:orderh} which boils down to the transitivity of the order $\mle$ which holds by Lemma~\mref{lem:extension}.\meqref{i:ext1}.

\meqref{i:ten2}. In the case of $k=2$, by Lemma \mref{lem:extension}.\meqref{i:ext3} we obtain
$$ [\vec u_2]\ot [\vec u_2]\nmle [\vec v_1]\ot [\vec u_2] \nmle [\vec v_1]\ot [\vec v_2].$$
The general case follows from an induction.

\end{proof}

We also have
\begin{lemma}\mlabel{lem:deltakeeporder}
Suppose $[\vec u]\otimes [\vec v]\nmle [\vec s]$ with $w([\vec u])+w([\vec v])=w([\vec s])$. Then for $0<i\le k=\dep([\vec s])$, we have
$$\text{either } (\id \cks p_i)([\vec u]\otimes [\vec v])=0 \text{ or } (\id \cks p_i)([\vec u]\otimes [\vec v])\nmle p_i([\vec s]),
$$
and
$$
\text{either } (p_i\otimes \id)([\vec u]\otimes [\vec v]) =0 \text{ or } (p_i\otimes \id)([\vec u]\otimes [\vec v])\nmle p_i([\vec s]).
$$

\end{lemma}

\begin{proof} Let $\vec u=(u_1, \cdots, u_m)$, $\vec v=(v_1, \cdots, v_n)$, $\vec s=(s_1, \cdots, s_k)$, 
and
$$[\vec t]=[t_1, \cdots, t_{m+n}]:=[\vec u, \vec v]=[u_1, \cdots, u_m, v_1,\cdots, v_n].$$
Then $[\vec t]\nmle [\vec s]$. So by Lemma \mref {lem:pkeepsorder}, we have: either $p_i([\vec u, \vec v ])=0$ or
$$p_i([\vec u, \vec v ])\nmle p_i([\vec s]). $$
We apply it to prove the two claims of the lemma by dividing into five cases.
\begin{itemize}
  \item [{\bf Case 1:}]
 Let $0< i\le m$. Then
$$(\id \cks p_i)([\vec u]\otimes [\vec v])=[\vec u]\otimes p_{i-m}([\vec v])=0,
$$
and
$$(p_i\otimes \id)([\vec u]\otimes [\vec v])=p_i([\vec u])\otimes [\vec v]\nmle [p_i([\vec u]), \vec v]=p_i([\vec u, \vec v])\nmle p_i([\vec s]).
$$
These give the desired inequalities thanks to the transitivity of $\nmle$ in Lemma~\mref{lem:exttensor}.(\mref{i:ten3}), which will be used repeatedly later without specifying.

\item [{\bf Case 2:}] Let $i=m+1$. Then
 \begin{equation*}
 \begin{split}
 (\id \cks p_{m+1})([\vec u]\otimes [\vec v])=&[\vec u]\otimes p_{1}([\vec v])=[\vec u]\otimes v_1 [v_1+1, v_2,\cdots, v_n]\\
 \nmle& [\vec u, v_1+1, v_2,\cdots, v_n]=[t_1,\cdots,t_m, t_{m+1}+1,t_{m+2}\cdots, t_{m+n}]\\
=&p_{m+1}([\vec u, \vec v])\nmle p_i([\vec s]).
 \end{split}
 \end{equation*}
 \begin{equation*}
 \begin{split}
 (p_{m+1}\otimes \id)([\vec u]\otimes [\vec v])=&-\delta_m([u_1,\cdots, u_m])\otimes [\vec v]\\
     =&-\sum_{i=1}^m u_i[u_1,\cdots,u_{i-1}, u_i+1,u_{i+1},\cdots,u_m]\otimes [\vec v]\\
     \nmle &\sum_{i=1}^m t_i[t_1,\cdots,t_{i-1}, t_i+1,t_{i+1}\cdots, t_{m+1},\cdots, t_{m+n}].
 \end{split}
 \end{equation*}
For each $1\le i\le m$, there is
 $$[t_1,\cdots,t_{i-1}, t_i+1,t_{i+1},\cdots, t_{m+1},\cdots, t_{m+n}]\mls[t_1,\cdots,t_m, t_{m+1}+1,t_{m+2},\cdots, t_{m+n}]=p_{m+1}([\vec u,\vec v]).
 $$
So
 $$(p_{m+1}\otimes \id)([\vec u]\otimes [\vec v])\nmle p_{m+1}([\vec u,\vec v])\nmle p_i([\vec s]).
 $$
 \item [{\bf Case 3:}] Let $m+1<i\le m+n$. Then
\begin{equation*}
\begin{split}
(\id \cks p_i)([\vec u]\otimes [\vec v])=&[\vec u]\otimes p_{i-m}([\vec v])\nmle [\vec u,  p_{i-m}([\vec v])]=p_i([\vec u, \vec v])\nmle p_i([\vec s]),
\end{split}
\end{equation*}
$$(p_i\otimes \id)([\vec u]\otimes [\vec v])=p_i([\vec u])\otimes [\vec v]=0.
$$
\item [{\bf Case 4:}] Let $i=m+n+1$. Then
\begin{equation*}
\begin{split}
(\id \cks p_{m+n+1})([\vec u]\otimes [\vec v])=&[\vec u]\otimes \Big(-\delta_n([\vec v])\Big)\nmle [\vec u, \delta_n([\vec v])]\\
=&\sum_{j=m+1}^{m+n}t_j[t_1,\cdots,t_m,\cdots,t_{j-1},t_j+1,t_{j+1},\cdots,t_{m+n}]\\
\nmle&[s_1,\cdots,s_{m+n},s_{m+n+1}+1,\cdots,s_k]=p_{m+n+1}([\vec s]),
\end{split}
\end{equation*}
$$(p_i\otimes \id)([\vec u]\otimes [\vec v])=0.
$$
 \item [{\bf Case 5:}] If $i>m+n+1$, then
 $$(\id \cks p_i)([\vec u]\otimes [\vec v])=[\vec u]\otimes p_{i-m}([\vec v])=0.
$$
$$(p_i\otimes \id)([\vec u]\otimes [\vec v])=p_i([\vec u])\otimes [\vec v]=0.
$$
\end{itemize}
This completes the proof.
\end{proof}

\subsection{Hopf algebra isomorphisms} The coproduct $\Delta _{\ge 1}$ on $\calhd$ has a special property as follows, thanks to the order and relation introduced in the previous subsection.
Recall that the coproduct $\Deltaa$ is homogeneous in the sense that,
for $[\vec s]\in H_n$, there is
\begin{equation}
\mlabel{e:coprodexp}
\Delta _{\ge 1}([\vec s])=\sum _{\vec u, \vec v}a_{\vec u, \vec v}^{\vec s}[\vec u]\otimes [\vec v]
\end{equation}
where $\{[\vec u]\otimes [\vec v]\}_{\vec u, \vec v}$ are the standard basis of $\bigoplus _{p+q=n}\Q H_p\otimes \Q H_q$.
Then the pure homogeneous tensors $[\vec s]$ and $[\vec u]\ot [\vec v]$ have the same weight as defined in Definition\,\mref{def:orderten}.

\begin{prop}
\mlabel{pp:Order} For $[\vec s]\in H_n$, there is
$$\Deltaa([\vec s]) \nmle [\vec s].$$
More precisely, if $a_{\vec u, \vec v}^{\vec s}\not =0$, then
$[\vec u]\otimes [\vec v]\nmle [\vec s].$
\end{prop}

\begin{proof}
	For $[\vec s]=[1_k]$, we have
	$$\Deltaa([1_k])=\sum\limits_{j=0}^{k}[1_j]\otimes [1_{k-j}],$$
	and by definition,
	$$[1_j]\otimes [1_{k-j}]\nmle [1_k].
	$$
Thus we have the inequality.
	
In general, by Eq.\,\meqref {eq:Fraction}, we have
	$$[s_1, \cdots, s_i, \cdots, s_k]=\frac {p_1^{s_1-1}\cdots p_k^{s_k-1}}{(s_1-1)!\cdots (s_k-1)!}[1_k].
	$$
So
\begin{align*}
&\Delta _{\ge 1}[s_1, \cdots, s_i, \cdots, s_k]=\Delta _{\ge 1}\Big(\frac {p_1^{s_1-1}\cdots p_k^{s_k-1}}{(s_1-1)!\cdots (s_k-1)!}[1_k]\Big)
\\
=&\frac {(\id\cks p_1+p_1\otimes \id)^{s_1-1}\cdots (\id\cks p_k+p_k\otimes \id)^{s_k-1}}{(s_1-1)!\cdots (s_k-1)!}(\Delta _{\ge 1}[1_k])
\\
=&\sum _{j=1}^k \frac {(\id\cks p_1+p_1\otimes \id)^{s_1-1}\cdots (\id\cks p_k+p_k\otimes \id)^{s_k-1}}{(s_1-1)!\cdots (s_k-1)!}([1_j]\otimes [1_{k-j}]).
\end{align*}

For fixed $1\le j\le k$, there is
	$$[1_j]\otimes [1_{k-j}]\nmle [1_k].
	$$
Applying Lemma \mref {lem:deltakeeporder} repeatedly, we obtain
	$$\Big((\id\cks p_1+p_1\otimes \id)^{s_1-1}\cdots (\id\cks p_k+p_k\otimes \id)^{s_k-1}\Big)([1_j]\otimes [1_{k-j}])\nmle (p_1^{s_1-1}\cdots p_k^{s_k-1})([1_k])
	$$
	$$=(s_1-1)!\cdots (s_k-1)![s_1, \cdots, s_k].
	$$
	This give the desired inequality.
\end{proof}

For a combinatorial Hopf algebra
$(\calhd, \shapb, \Delta _{\ge 1}, \chi)$, we apply the general construction in~\mcite{ABS06} and display the following explicit description of $\Psi_\chi$ from Eq.\,\meqref{e:psiform} for later applications.
\begin{equation}
	\mlabel {eq:Psi}
	\Psi_\chi([\vec s])=\sum_{\alpha\models n} \chi_\alpha([\vec s]) [\vec \alpha],
\end{equation}
where $\alpha=(\alpha_1, \cdots, \alpha_m)$, $[\vec \alpha]=[\alpha _1, \cdots, \alpha _m]$, and $\chi_{\alpha}$ is the composition
\begin{equation}
	\chi_\alpha: \calhd\xrightarrow{\Delta _{\ge 1}^{(m - 1)}}(\calhd)^{\otimes m}\xrightarrow{\pi_\alpha} \Q H_{\alpha_1}\otimes\cdots\otimes \Q H_{\alpha_m}\xrightarrow{\chi^{\otimes m}}\mathbb{Q}.
	\mlabel{eq:chialpha}
\end{equation}
where
$$
\Deltaa^{(0)} := \text{id}_{H}, \quad
\Deltaa^{(1)} := \Deltaa,
$$
$$
\Deltaa^{(m)} := (\Deltaa \otimes \text{id}^{\otimes (m-1)}) \circ \Deltaa^{(m-1)}=(\Deltaa^{(m-1)}\otimes \id)\circ \Deltaa,
$$
and $\pi_\alpha:=\ot_i \pi_{\alpha_i}$ is the tensor product of the canonical projections $\pi_{\alpha_i}: \calhd\to \Q H_{\alpha_i}$ onto the homogeneous components.

Since $\alpha_1+\cdots+\alpha_k=n$, we have $\Psi_\chi: \Q H_n\longrightarrow \Q H_n$. So $\Psi _\chi$ is homogeneous.

\begin{prop} \mlabel{prop:Psiisomorphism} For a combinatorial Hopf algebra
	$(\calhd, \shapb, \Delta _{\ge 1}, \chi)$ and $n\geq 0$, the restriction of the induced combinatorial Hopf algebra homomorphism
	$\Psi _\chi |_{\Q H_n} $ is upper triangular with respect to the standard basis $H_n$ according to the order in Eq\,\meqref{e:order1}.
\end{prop}

\begin{proof}
For a given $[\vec s]\in H_n$, we only need to prove that, if $\vec \alpha\not=\vec s$ and $\chi _\alpha ([\vec s])\not=0$ in Eq. (\mref{eq:Psi}), then $[\vec \alpha]\mls[ \vec s]$ .

By Proposition \mref {pp:Order} and Lemma \mref{lem:exttensor}, when we write
$$\Delta _{\ge 1}^{(m-1)}([\vec s])=\sum a_{\vec u_1, \cdots, \vec u_m} [\vec u_1]\otimes \cdots \otimes [\vec u _m],
$$
here $[\vec u_i]$ can be ${\bf 1}=[\emptyset]$, then  $a_{\vec u_1, \cdots, \vec u_m}\not =0$ implies
\begin{equation}
\mlabel{e:ums}
[\vec u_1]\otimes \cdots \otimes [\vec u _m]\nmle [\vec s].
\end{equation}
	
The projection of $\Delta _{\ge 1}^{(m-1)}([\vec s])\subseteq (\calhd)^{\ot m}$ to $\Q H_{\alpha_1}\otimes\cdots\otimes \Q H_{\alpha_m}$ via $\pi_\alpha$ annihilates all $m$-tensors $[\vec u_1]\otimes \cdots \otimes [\vec u _m]$ except those
	with weights
	$$w([\vec u_1])=\alpha _1, \cdots, w([\vec u_m])=\alpha _m.
	$$
When these inequalities hold, by Lemma~\mref{lem:extension}.\meqref{i:ext2}, we have
	$$[\alpha _1]\nmle [\vec u_1], \cdots, [\alpha _m]\nmle [\vec u_m].
	$$
Thus
\begin{align*}
[\vec \alpha ]&\nmle [\alpha_1]\ot \cdots \ot [\alpha_m] \quad \text{(by Lemma~\mref{lem:exttensor}.\meqref{i:ten1})} \\
&  \nmleq [\vec u_1] \ot \cdots \ot [\vec u_m]
\quad \text{(by Lemma~\mref{lem:exttensor}.\meqref{i:ten2})}\\
&\nmleq [\vec s]\quad \text{(by Eq.\meqref{e:ums})}.
\end{align*}
So $[\vec \alpha]\nmle [\vec s]$ and hence $[\vec \alpha]\mle [\vec s]$. By Lemma~\mref{lem:extension}.\meqref{i:ext1}, $\mle $ is a well order on $H_n$. So from $[\vec \alpha] \mle [\vec s]$ and $[\vec \alpha] \ne [\vec s]$, we obtain $[\vec \alpha]\mls[\vec s]$, as needed.
\end{proof}

\begin{coro}
For the combinatorial Hopf algebra $(\calhd, \shapb, \Delta _{\ge 1}, \chi)$ in Proposition~\mref{prop:Psiisomorphism} and scalar $s\in \Z_{\ge 1}$, we have
	$$\Psi _\chi ([s])=\chi ([s])[s].
	$$
\end{coro}
\begin{proof}
This follows from Eq.\,\meqref{eq:Psi} and the fact that $[s]$ is the smallest among its compositions thanks to Lemma~\mref{lem:extension}.\meqref{i:ext2}. 	
\end{proof}
	
\begin{theorem}
	\mlabel{t:chiisom} The combinatorial Hopf algebra homomorphism $\Psi _\chi$ is an isomorphism  if and only if $\chi ([s])\not =0$ for every $s\in \Z_{\ge 1}$.
	\end{theorem}
\begin{proof}
	We already know that $\Psi _\chi |_{\Q H_n}$ is upper triangular. Now let us find the coefficient $\chi_{[\vec s]}$ of $[\vec s]$ in $\Psi_\chi ([\vec s])$.
	
	Let $[\vec s]=[s_1, \cdots, s_k]$. If $k=1$, then by definition,
	$$\Delta _{\ge 1}^{(0)}([s_1])=[s_1],
	$$
	$$\chi _{[s_1]}([s_1])=\chi ([s_1]).
	$$
	If $k>1$,
	then
	$$\Delta^{(k-1)} _{\ge 1}([\vec s])=(\Delta^{(k-2)} _{\ge 1}\otimes \id)(\Delta _{\ge 1}([\vec s]))
	$$
	and
	$$\Delta _{\ge 1}([\vec s])=\sum _{j=0}^k \frac {(\id\cks p_1+p_1\otimes \id)^{s_1-1}\cdots (\id\cks p_k+p_k\otimes \id)^{s_k-1}}{(s_1-1)!\cdots (s_k-1)!}([1_j]\otimes [1_{k-j}]).
	$$
In this sum, consider a term corresponding to $j=0,\cdots, k-2$. Then $\dep([1_{k-j}])\ge 2$. Note that $p_i$ does not change the depth. So the second tensor factor of the term has depth greater or equal to 2. Then the first tensor factor has depth less or equal to $k-2$. This means that when $\Delta^{(k-2)} _{\ge 1}$ is applied to this term, at least one of the resulting $k-1$ tensor factors has to be $\bf 1$. This also means that, in their contributions to
$\Delta^{(k-1)} _{\ge 1}([\vec s])$, the $k$-tensors
	$$[\vec u_1]\otimes \cdots \otimes [\vec u_k]
	$$
has at least one tensor factor $[\vec u_i]=[\emptyset]={\bf 1}$.
Hence it is annihilated by the projection
	$$\pi _{[\vec s]}: (\calhd) ^{\otimes k}\to \Q H_{s_1}\otimes\cdots\otimes \Q H_{s_k}$$
and hence by $\chi_{[\vec s]}$ in Eq.~\meqref{eq:chialpha}.	
The same holds for the terms corresponding to $j=k$.
Thus the only contribution to $\chi_{[\vec s]}$ comes from the term in the sum with $j=k-1$.
	
In summary, for
	$j=0, \cdots, k-2$ or $j=k$, the contributions in $\Delta^{(k-1)} _{\ge 1}([\vec s])$ from $[1_j]\otimes [1_{k-j}]$ are in $\ker \pi _{[\vec s]}$.
So
\begin{align*}
&\Delta^{(k-1)} _{\ge 1}([\vec s])\\
&\equiv (\Delta^{(k-2)} _{\ge 1}\otimes \id)\Big(\frac {(\id\cks p_1+p_1\otimes \id)^{s_1-1}\cdots (\id\cks p_k+p_k\otimes \id)^{s_k-1}}{(s_1-1)!\cdots (s_k-1)!}\Big)([1_{k-1}]\otimes [1])\bmod \ker \pi _{[\vec s]}.
\end{align*}
Since the last component of $\pi _{[\vec s]}$ is the projection to $\Q H_{s_k}$, only $\id \cks p_k$ in the above expression can increase the last component in $[1_{k-1}]\otimes [1]$ to $w \otimes [s_k]$ for some multi-tensor $w$. Thus the previous expression equals to
\begin{align*}
&(\Delta^{(k-2)} _{\ge 1}\otimes \id)\Big(\frac {(\id\cks p_1+p_1\otimes \id)^{s_1-1}\cdots (\id\cks p_{k-1}+p_{k-1}\otimes \id)^{s_{k-1}-1}}{(s_1-1)!\cdots (s_{k-1}-1)!}\Big)([1_{k-1}]\otimes [s_k]))
\bmod \ker \pi _{[\vec s]}.
\end{align*}
Continuing this computation, for $i=1, \cdots, k-1$, there is
	$$(\id \cks p_i)([1_{k-1}]\otimes [s_k] )=0,$$
and	so
\begin{equation*}
\begin{split}
	\Delta^{(k-1)} _{\ge 1}([\vec s])
\equiv &(\Delta^{(k-2)} _{\ge 1}\otimes \id)\Big(\frac {(p_1\otimes \id)^{s_1-1}\cdots (p_{k-1}\otimes \id)^{s_{k-1}-1}}{(s_1-1)!\cdots (s_{k-1}-1)!}\Big)([1_{k-1}]\otimes [s_k])) \bmod \ker \pi _{[\vec s]}\\
\equiv&(\Delta^{(k-2)} _{\ge 1}\otimes \id)([s_1,\cdots, s_{k-1}]\otimes [s_k]) \bmod \ker \pi _{[\vec s]}\\
\equiv &\Delta^{(k-2)} _{\ge 1}([s_1,\cdots, s_{k-1}])\otimes [s_k] \bmod \ker \pi _{[\vec s]}.
\end{split}
\end{equation*}
Iterating this process, we eventually obtain
$$\Delta^{(k-1)} _{\ge 1}([\vec s])
		\equiv [s_1]\otimes\cdots \otimes [s_k] \bmod \ker \pi _{[\vec s]}.
$$	
Therefore
	$$\chi_{[\vec s]} ([\vec s])=\chi([s_1])\cdots\chi([s_{k}]).$$
Thus, $\Psi_\chi |_{\Q H_n}$ is upper triangular with diagonal element $\chi([s_1])\cdots\chi([s_{k}])$ at the spot $[s_1, \cdots, s_k]$.
This shows that $\Psi _\chi|_{\Q H_n}$ is invertible if and only $\chi ([s])\not =0$  for every $s\in \Z_{\ge 1}$.
\end{proof}

Now let
$$\mathrm{Hom}\big((\calhd, \shapb, \Delta _{\ge 1}), (\calhd, \ast, \Delta _{\rm  dec})\big)$$
denote the set of graded Hopf algebra homomorphisms of degree $0$ between the graded connected Hopf algebras $(\calhd, \shapb, \Delta _{\ge 1})$ and $(\calhd, \ast, \Delta _{\rm  dec})$, and let
$$\Isom\big((\calhd, \shapb, \Delta _{\ge 1}),(\calhd\ast, \Delta _{\rm  dec})\big)$$
denote set of graded Hopf algebra isomorphisms of degree 0.

\begin{theorem} \mlabel{t:psiisom}
The map
$$\Psi: \Big\{\chi: (\calhd, \shapb) \to \Q\, \Big|\ \chi  \text{ is a character}\Big\}\to \mathrm{Hom}\big((\calhd, \shapb, \Delta _{\ge 1}), (\calhd, \ast, \Delta _{\rm  dec})\big),
$$
$$\chi \mapsto \Psi _\chi,
$$
is bijective.
Furthermore, the bijection restricts to a bijection
$$\Psi: \{\chi: (\calhd, \shapb) \to \Q |\ \chi ([s])\not =0, \forall s \in \Z _{\ge 1}\}\to \Isom\big((\calhd, \shapb, \Delta _{\ge 1}), (\calhd,\ast, \Delta _{\rm  dec})\big).
$$
\end{theorem}

\begin{proof}
By Corollary\,\mref{coro:IndHom}, the map $\Psi$ is an injective map. On the other hand, for a Hopf algebra homomorphism $F:(\calhd, \shapb, \Delta _{\ge 1})\to  (\calhd,\ast, \Delta _{\rm  dec})$, the composition $\chi'_Q\circ F$ is a character on $(\calhd, \shapb, \Delta _{\ge 1})$. Since $\chi'_Q\circ \Psi_\chi=\chi$ also holds, the uniqueness in Corollary\,\mref{coro:IndHom} gives $F=\Psi_\chi$, proving the surjectivity of $\Psi$.

The second conclusion then follows from  Theorem\,\mref{t:chiisom} by restrictions.
\end{proof}

\subsection{An explicit isomorphism between the MZV Hopf algebras}
We end the paper by showing that there is an isomorphism from $(\calhd, \shapb, \Delta _{\ge 1})$ to $(\calhd, \ast, \Delta _{\rm dec})$. By Theorem~\mref{t:psiisom}, we just need to provide a character
$$\chi: \calhd \to \Q
$$
with $\chi ([n])\not =0$ for all $n\in \Z _{\ge 1}$.

\begin{theorem}
\begin{enumerate}
\item
The linear map
\begin{equation}
\mlabel{e:char0}
\begin{split}
\mchi: &(\calhd, \shapb)\longrightarrow \Q, \\
[s_1,\cdots, s_k]&\mapsto \frac{1}{(s_1+\cdots+s_k)!}, \quad \mchi({\bf1})=1
\end{split}
\end{equation}
is a character.
\mlabel{i:chachi1}
\item
Correspondingly, the homomorphism $\Psi_{\mchi}$ is a Hopf algebra isomorphism from
$(\calhd, \shapb, \Delta _{\ge 1})$ to $(\calhd, \ast, \Delta _{\rm dec})$.
\mlabel{i:chachi2}
\end{enumerate}
\mlabel{t:shachi}
\end{theorem}

\begin{proof}
\meqref{i:chachi1}
Let $[\vec s]$ and $[\vec t]$ be two basis elements of $\calhd$, with weight $m$ and $n$ respectively.
The shuffle product $[\vec s]\shapb[\vec t]$ is the sum of shuffles of $[\vec s]$ and $[\vec t]$, with monic coefficients. Since each shuffle is a permutation of $[\vec s]$ and $[\vec t]$, the weight of a shuffle is the sum $m+n$. Furthermore, there are $\binom{m+n}{m}$ shuffles. Thus
$$ \mchi ([\vec s]\shapb [\vec t]) = \binom{m+n}{m} \frac{1}{(m+n)!} = \frac{1}{m!} \frac{1}{n!}= \mchi([\vec s])\mchi([\vec t]),$$
as needed.

Then \meqref{i:chachi2} follows from Theorem\,\mref{t:psiisom}.
\end{proof}

\begin{example}
	\mlabel{ex:psiexam}
We give some values of the Hopf algebra isomorphism $\Psi_0:=\Psi_{\mchi}$. Compare with the Hoffman-Newman-Radford isomorphism $\exp$ in Eq.\,\meqref{e:hoffexp}.
\begin{align*}
	\Psi_0([n]) &= \mchi([n])[n] =\frac{1}{n!} [n], n\geq 1,\\
	\Psi_0([1,1]) &= \mchi([1,1])[2] + \mchi^2([1])[1,1]
	=\frac{1}{2}[2]+[1,1], \\
	\Psi_0([1,2]) &= \mchi([1,2])[3] + \mchi([1])\mchi([2])[1,2] - \mchi([2])\mchi([1])[2,1]
	=\frac{1}{6}[3]+\frac{1}{2}[1,2]-\frac{1}{2}[2,1],\\
	\Psi_0([2,2]) &= \mchi([2,2])[4] + \mchi^2([2])[2,2] - 2\mchi([3])\mchi([1])[3,1]=\frac{1}{4!}[4]+\frac{1}{4}[2,2]-\frac{1}{3}[3,1].
\end{align*}
\end{example}

Then composing $\Psi_0$ with $\log$ in Eq.\,\meqref{e:hofflog} gives the following conclusion.
\begin{theorem}
\mlabel{t:hoffiso}
\begin{enumerate}
\item There is a Hopf algebra isomorphism
\begin{equation}
\mlabel{e:hoffiso}
	\log \circ \Psi_0: (\calhd,\shapb,\Delta_{\ge 1}) \to (\calhd,\shap,\Delta_{\rm dec}).
\end{equation}
\mlabel{i:hoffiso1}
\item
We have the following commutative diagram.
\begin{equation}
\mlabel{d:isorel}
\begin{split}
\xymatrix{
	(\calhd, \shapb, \Delta_{\ge 1}) \ar[ddr]_{\log\circ \Psi_{0}} \ar[dr]^{\chi_{_0}} \ar[rr]^{\Psi_{0}}&&(\calhd, \ast, \Delta_{\text{dec}})\ar[dl]_{\chi_Q'}                  \\
	& \Q &
	\\
	& (\calhd, \shap, \Delta_{\text{dec}})\ar[u]^{\chi_H}\ar[uur]_{\exp} &}
\end{split}
\end{equation}
Here the characters $\chi_Q', \chi_H$ and $\chi_{{_0}}$ are given in Eq.\,\meqref{eq:zetaQdefinitionQSym}, \meqref{e:charh} and \meqref{e:char0} respectively.
\mlabel{i:hoffiso2}
\end{enumerate}
\end{theorem}
\begin{proof}
\meqref{i:hoffiso2} To verify the commutativity of each triangle, by Theorem\,\mref{t:shachi}.\meqref{i:chachi2}, there is $\chi_{_0}=\chi_Q'\circ \Psi_0$.  By Example \mref{exm:Hoffiso}, we have $\chi_H=\exp\circ \chi_Q'$, that is, $\chi_H\circ \log=\chi_Q'$. Thus by Item\,\meqref{i:hoffiso1}, we have
$\chi_H\circ\log \circ \Psi_0=\chi_{_0}.$
\end{proof}

We end this study with some values of the Hopf algebra isomorphism $\log \circ \Psi_0$ and leave more detailed analysis to a future work.
\begin{example}
Continuing with Example\,\mref{ex:psiexam}, we have
{\small
\begin{align*}
\log \circ \Psi_0([n])&=\log\bigg(\frac{1}{n!}[n]\bigg)=\frac{1}{n!}[n],\\
\log \circ \Psi_0([1, 1])&=\log\bigg(\frac{1}{2}[2]+[1,1]\bigg)=\frac{1}{2}[2]+[1,1]-\frac{1}{2}[2]=[1,1],\\
\delete{
\log \circ \Psi_0([2, 1])&=\log(\frac{1}{6}[3]+\frac{1}{2}[2,1])=\frac{1}{6}[3]+\frac{1}{2}([2,1]-\frac{1}{2}[3])=\frac{1}{2}[2,1]-\frac{1}{12}[3],\\
}
\log \circ \Psi_0([1, 2])&=\log\bigg(\frac{1}{6}[3]+\frac{1}{2}[1,2]-\frac{1}{2}[2,1]\bigg)\\
&=\frac{1}{6}[3]+\frac{1}{2}([1,2]-\frac{1}{2}[3])-\frac{1}{2}([2,1]-\frac{1}{2}[3])\\
&=\frac{1}{6}[3]+\frac{1}{2}[1,2]-\frac{1}{2}[2,1],\\
\delete{
\log \circ \Psi_0([1, 1, 1])&=\log( \frac{1}{6}[3]+\frac{1}{2}[1,2]+\frac{1}{2}[2,1]+[1,1,1])\\
&=\frac{1}{6}[3]+\frac{1}{2}([1,2]-\frac{1}{2}[3])+\frac{1}{2}([2,1]-\frac{1}{2}[3])+\frac{1}{6}[3]-\frac{1}{2}[2,1]-\frac{1}{2}[1,2]+[1,1,1]\\
&=-\frac{1}{6}[3]+[1,1,1],\\
\log \circ \Psi_0([3, 1])&=\log(\frac{1}{24}[4]+\frac{1}{6}[3,1])\\
&=\frac{1}{24}[4]+\frac{1}{6}([3,1]-\frac{1}{2}[4])\\
&=\frac{1}{6}[3,1]-\frac{1}{24}[4],\\
}
\log \circ \Psi_0([2,2])&=\log\bigg(\frac{1}{24}[4]+\frac{1}{4}[2,2]-\frac{1}{3}[3,1]\bigg)\\
&=\frac{1}{24}[4]+\frac{1}{4}([2,2]-\frac{1}{2}[4])-\frac{1}{3}([3,1]-\frac{1}{2}[4])\\
&=\frac{1}{12}[4]+\frac{1}{4}[2,2]-\frac{1}{3}[3,1].
\end{align*}
}
\end{example}

\noindent
{\bf Acknowledgments.} This research is supported by
the NSFC (12471062, 12261131498).

\noindent
{\bf Declaration of interests. } The authors have no conflict of interest to declare.

\noindent
{\bf Data availability. } Data sharing is not applicable
as no data were created or analyzed.

\end{document}